\documentclass[11pt,reqno]{amsart}

\usepackage{enumerate,url,amssymb,  mathrsfs}
\usepackage{graphicx}
\usepackage{esint,comment}
\usepackage{hyperref}
\pdfoutput=1

\newtheorem{theorem}{Theorem}[section]
\newtheorem{lemma}[theorem]{Lemma}
\newtheorem*{lemma*}{Lemma}
\newtheorem{proposition}[theorem]{Proposition}
\newtheorem{corollary}[theorem]{Corollary}

\theoremstyle{definition}

\newtheorem{example}[theorem]{Example}
\newtheorem{conjecture}[theorem]{Conjecture}

\theoremstyle{remark}
\newtheorem{remark}[theorem]{Remark}
\newtheorem{problem}[theorem]{Problem}

\numberwithin{equation}{section}

\newcommand{\refeq}[1]{(\ref{#1})}

\newcommand{\C}{\mathbb{C}}

\newcommand{\D}{\partial}
\newcommand{\DD}{\mathbb{D}}

\newcommand{\R}{\mathbb{R}}

\DeclareMathOperator{\re}{Re}

\DeclareMathOperator{\im}{Im}

\DeclareMathOperator{\Hdim}{dim_{\mathcal H} }

\def\XXint#1#2#3{{\setbox0=\hbox{$#1{#2#3}{\int}$}
\vcenter{\hbox{$#2#3$}}\kern-.5\wd0}}

\begin{document}
\baselineskip6mm
\vskip0.4cm
\title{Quasidisks and twisting of the Riemann map}

\author[I. Prause]{Istv\'an Prause}
\address{Department of Mathematics and Statistics, University of Helsinki,
         P.O. Box 68, FIN-00014, Helsinki, Finland}
\email{istvan.prause@helsinki.fi}

\thanks{The author was supported by the Academy of Finland grants 1266182, 1273458 and 1303765.}




\begin{abstract}
Consider a conformal map from the unit disk onto a quasidisk. We determine a range of critical complex powers with respect to which the derivative is integrable. 
The results fit into the picture predicted by a circular analogue of Brennan's conjecture.
\end{abstract}

\maketitle


\section{Introduction}\label{se:introduction}

Let $f \colon \DD \to \Omega$ be a bounded conformal map defined on the unit disk $\DD$. The {\em integral means spectrum} of $f$ is the function
\begin{equation}
\beta_f(t)=\limsup_{r \to 1} \frac{\log \int_{|z|=r} \left| f'(z)^t \right| |d z|}{\log \frac{1}{1-r}}, \qquad t \in \C.
\end{equation}
The complex power $(f')^t$ is defined in terms of the complex logarithm. As the derivative $f'$ is never zero in $\DD$, we can select a unique single-valued branch of $\log f'(z)$ by requiring that $\arg f'(0) \in [0,2\pi)$. For real $t$ (positive or negative), integral means measure boundary expansion and compression associated with a given conformal map. Allowing $t$ to be complex amounts to addressing rotational phenomena, as well.

The {\em universal integral means spectrum} is obtained by taking the supremum over all bounded conformal maps $f$,
\[ B(t)=\sup_{f} \beta_f(t), \qquad t \in \C.
\]
We will actually work with conformal maps with quasiconformal extensions and correspondingly introduce
for any $0\leqslant k <1$,
\[ B_k(t)=\sup \{ \beta_f(t) \colon \mbox{ $f$ has a $k$-quasiconformal extension to $\C$} \}, \quad t \in \C.
\]
The quasiconformal extension means here that $f$ extends to a  $W^{1,2}_{loc}$-homeo\-morphism $f \colon \C \to \C$ with the distortion inequality
\[ |\bar \D f (z)| \leqslant k |\D f(z)|, \quad \mbox{for a.e. $z \in \C$}.
\]

Brennan's classical conjecture \cite{brennan,jones} states that $B(-2)=1$. Note that the identity $B(2)=1$ is rather easy (follows readily from the finite area of $\Omega$). On the other hand, there is no real evidence that $B(t)$ is even function. Even a stronger radial invariance was hinted by Becker and Pommerenke in \cite{BP} who proposed the following.

\begin{conjecture}[Circular Brennan's conjecture]
\label{conj:circular-brennan}
\begin{equation}
B(t)=1, \qquad |t|=2.
\end{equation}
\end{conjecture}

In this work we provide some evidence for this conjecture by establishing partial radial invariance for the function $B_k$.
\begin{theorem} For every $0<k<1$ we have
\label{thm:main}
\begin{equation}
B_k(t) =1, \quad |t|=\frac{2}{k} \qquad \re  t \geqslant 2.
\end{equation}
\end{theorem}

We believe that in fact the identity
\begin{equation}
\label{eq:qcbrennan}
B_k(t) =1, \quad |t|=\frac{2}{k}
\end{equation}
holds without restriction on the real part of $t$. This is basically equivalent to Conjecture \ref{conj:circular-brennan}. On one hand, in view of the fractal approximation principle \cite{makarov99}, we have $B(t)=\sup_{k<1} B_k(t)$.
In the other direction, we sketch an argument along the lines of \cite{BP}. We first embed our map $f$
into a standard holomorphic motion $f_\lambda$, with $f_k=f$ and $f_0(z)=z$. We then consider the subharmonic function for a fixed radius $r<1$.
\[ \lambda \mapsto \int_{|z|=r} \left| f_\lambda'(z)^{2/\lambda} \right| |dz|, \quad \lambda \in \DD.
\]
Conjecture \ref{conj:circular-brennan} controls the growth rates in $r$ and the maximum principle leads to the upper bound in \eqref{eq:qcbrennan}. The lower bounds are easy, see Remark \ref{rmk:trivialbounds}. The only issue with the above sketch is that we need uniform estimates (independent of $f_\lambda$) of the growth rates. Actually, Becker and Pommerenke postulated this stronger uniform version in \cite{BP}.

The approach of this paper is similar in spirit to the argument above: to use the maximum principle to upgrade a priori bounds with the help of holomorphic dependence. The a priori information we will use is the simple bound $B(2)=1$ (and related area bounds for quasiconformal maps). 
We amplify these using the holomorphic interpolation technique of \cite{AIPS,AIPS2}. 
Instead of a standard holomorphic motion, we consider here a holomorphic motion in two variables with certain symmetry properties. Conformality of the map is exploited at this point, it gives us room to change the conformal structure in terms of the second variable. This extra structure of the motion leads to an improvement compared to general  quasiconformal multifractal bounds of \cite{AIPS2}.
In this setting, the maximum principle is invoked in the form of a Nevanlinna-Pick interpolation problem on the bidisk. 

The special case $B_k(2/k)=1$ of Theorem \ref{thm:main} was proved in \cite{prause-smirnov}. Recently, Hedenmalm obtained \cite{hedenmalm15} a general bound (without restriction on $t$) of the form 
\[ B_k(t) \leqslant (1+7k)^2 \frac{k^2 |t|^2}{4}, \quad t \in \mathbb{C}.
\]
When specialised to the values of $t$ covered by Theorem \ref{thm:main} this is weaker by the factor $(1+7k)^2$.
Finally, let us note that the strongest conjecture in this area \cite{jones,prause-smirnov} of the form
\begin{equation*}
 B_k(t)=k^2 |t|^2/4, \qquad |t| \leqslant \frac{2}{k},
\end{equation*}
has recently been disproved by Ivrii \cite{ivrii}. Namely, he shows that there are better bounds asymptotically as $k \to 0$ and $k|t| \to 0$.
Ivrii's approach builds on the recent developments of \cite{aipp,hedenmalm16} and on \cite{BP}.
In this paper, we consider the opposite asymptotic regime, when $k|t| \approx 1$ and $k$ arbitrary. This regime is related to the following phase transition phenomenon.

\subsection{Phase transition}
Using the convexity of $B_k(t)$ and the easy bounds of Remark \ref{rmk:trivialbounds}, it follows that Theorem \ref{thm:main} can be extended to
\[ B_k(t) = k|t|-1, \quad \re t \geqslant k|t|,  \quad |t| \geqslant \frac{2}{k}.
\]
Consider an angle $\theta \in [0,2 \pi)$ with $\cos \theta \geqslant k$. It means that for large positive values of $t$, namely when $t \geqslant 2/k$, the function $B_k(te^{i\theta})$ is linear in $t$. From considerations of small values of $t$, in each of these directions there is a phase transition point where the function becomes strictly convex.
The existence of such a phase transition phenomenon for $B(t)$ in the negative direction was shown by Carleson and Makarov \cite{carleson-makarov} and later extended to arbitrary direction in \cite{binder}. As explained in \cite{carleson-makarov}, this  can be interpreted as the transition from isolated singularities to fractal distribution of singularities in the extremal contribution to integral means.

\paragraph{\em Acknowledgements}
It is my pleasure to thank H\aa kan Hedenmalm, Oleg Ivrii and Stas Smirnov for several discussions on topics related to the present paper.

\section{Pointwise bounds}

Classical distortion estimates in conformal mappings describe the optimal rate of growth for the derivative.
For instance, we have \cite[p.~66]{pommerenke-univalent}

\begin{proposition}
\label{prop:prop1}
For $f \in \mathcal{S}$, we have
\[ \left| \log \frac{z f'(z)}{f(z)} \right| \leqslant \log \frac{1+|z|}{1-|z|}.
\]
\end{proposition}

Here $\mathcal S$ denotes the class of univalent maps of the disk with normalizations $f(0)=0$, $f'(0)=1$.
Similarly, we will use notation $\mathcal S_k$ for maps in $\mathcal S$ admitting $k$-quasiconformal extension to $\mathbb{C}$. An immediate application of Lehto's majorant principle \cite[p.~77]{lehto} shows

\begin{proposition}
\label{prop:propk}
For $f \in \mathcal{S}_k$, we have
\[ \left| \log \frac{z f'(z)}{f(z)} \right| \leqslant k \log \frac{1+|z|}{1-|z|}.
\]
\end{proposition}

One can alternatively describe the pointwise bounds in terms of scaling and rotation exponents.
For this, let us introduce a few definitions.
For a simply connected domain $\Omega \subset \C$, consider the harmonic measure $\omega$ with a basepoint $w_0 \in \Omega$. The choice of the basepoint will be irrelevant in what follows. We say that the harmonic measure {\em scales with exponent} $\alpha >0$ at $x \in \D \Omega$ if  
\[ \lim_{r \to 0} \frac{\log \omega B(x,r)}{\log r} = \alpha.
\]
In a dual way, we will measure the rotation near $x$. Let $\Omega_r$ be the connected component of $\Omega \setminus B(x,r)$ containing $w_0$. 
We say that $\Omega$ {\em rotates at rate} $\gamma \in \R$ if 
\[ \lim_{r \to 0} \frac{ \inf_{w \in \partial \Omega_r \cap B(x,r)} \arg(w-x)}{\log r} = \gamma.
\]
The branch of the argument is selected so that $\arg(w_0-x) \in [0,2\pi)$.
In other words, $\gamma$ measures the rate of rotation (in the sense of \cite{AIPS2}) of how fast a curve should rotate around point $x$ in order to get to the $r$-neighbourhood of $x$ from $w_0$ within the domain $\Omega$.
Of course, in general, the limit in these scaling and rotation exponents need not exist. To rectify this, one could consider subsequential limits. In this case, it is important to require that we measure both scaling and rotation  simultaneously along the same subsequence. 

A well-known estimate of Beurling \cite[Corollary 9.3]{garnett-marshall} says that $\omega B(x,r) \leqslant C r^{1/2}$. In terms of the scaling exponent, this means that $\alpha$ (if exists) is at least $1/2$. A slightly more general version relates the scaling exponent to the twisting at a boundary point via the inequality, see \cite[Lemma 1]{binder}
\begin{equation}
\label{eq:beurling}
 \alpha \geqslant \frac{1}{2}(1+\gamma^2).
\end{equation}

\subsection{Examples}
The basic examples for extremal pointwise behaviour are more conveniently described as conformal maps of the upper half plane $\C_+=\{ z : \im z >0 \}$. Conjugating with an appropriate M\"obius transformation provides examples defined on the unit disk.
\begin{example}
\label{ex:basicexamples}
The {\em complex power map} (with the principal branch) $g_\sigma \colon \C_+ \to \C$, $g_\sigma(z)=z^\sigma$ defines a conformal map as long as $|\sigma-1| \leqslant 1$ and $\sigma \neq 0$.
\end{example}

With the notation $\sigma=\frac{1+i \gamma}{\alpha}$, $\alpha>0$, $\gamma \in \R$, the conformal map $g_\sigma$ maps the positive and negative half-lines of $\R$ to logarithmic spirals of rotation rate $\gamma$, while harmonic measure scales with exponent $\alpha$ at the origin. The condition $|\sigma-1| \leqslant 1$, $\sigma \neq 0$ is needed to ensure injectivity. Equivalently, this condition is exactly \eqref{eq:beurling}.

Observe that $g_{\sigma_0}$ has a $|\sigma_0-1|$-quasiconformal extension to the lower half plane $\C_-$. This follows readily from the $\lambda$-lemma \cite[Section 12.3]{AIMb}. Indeed, $g_{\sigma_0}$ embeds to the holomorphic family $g_\sigma$ parametrised by $\sigma-1 \in \DD$. In Section \ref{se:example} we describe the extensions and the corresponding holomorphic motion explicitly.

\begin{remark}
\label{rmk:trivialbounds}
The pointwise estimates of Proposition \ref{prop:prop1} and \ref{prop:propk} give the following \emph{trivial upper bounds} for integral means spectra
\[ B(t) \leqslant |t| \qquad B_k(t) \leqslant k |t| \qquad t \in \mathbb{C}.
\]
The {\em trivial lower bounds} are obtained by considering Example \ref{ex:basicexamples} and read as
\[ B(t) \geqslant |t|-1 \qquad B_k(t) \geqslant k |t|-1.
\] 
\end{remark}

\section{A three-point Nevanlinna-Pick problem on the bidisk}
\label{se:NP}

We are going to make use of the following three-point Nevanlinna-Pick problem on the bidisk. In what follows, $U=\{w \colon \re w>0 \}$ will denote the right-half plane.

\begin{problem}
\label{prob:threepoint}
Given $\lambda_0 \in \DD$, define the set $\mathcal{W}_{\lambda_0}$ as 
\[ \mathcal{W}_{\lambda_0}=\left\{ w \in \mathbb{C} \; | \; \Psi \colon \DD^2 \to U, \Psi(0,0)=1, \Psi(\lambda_0,0)=w, \Psi(0,\bar \lambda_0)= \bar w \right\},
\]
where $\Psi$ ranges over all holomorphic mappings of the bidisk $\Psi \colon \DD^2 \to U$. Describe $\mathcal{W}_{\lambda_0}$ in terms of the parameter $\lambda_0$.
\end{problem}

First, let us observe that $\mathcal{W}_{\lambda_0}$ only depends on $|\lambda_0|$. This can be seen by considering transformations of the form $\Psi(e^{i\theta} \lambda,e^{-i \theta} \eta)$.

\begin{lemma} 
\label{lemma:NP}
Let $|\lambda_0|=k<1$.
The solution of Problem \ref{prob:threepoint} is described by the convex hull of the union of two disks (see Figure \ref{fig:3point})
\begin{equation}
\label{eq:two-disks}
 \mathcal{W}_{\lambda_0}=\mathcal{W}_k=\text{Conv}\left( \left\{w \colon |w-1| \leqslant k \right\} \cup \left\{ w \colon \left| \frac{1}{w}-1 \right| \leqslant k \right\} \right).
\end{equation}
\end{lemma}

\begin{figure}
\includegraphics[width=6cm]{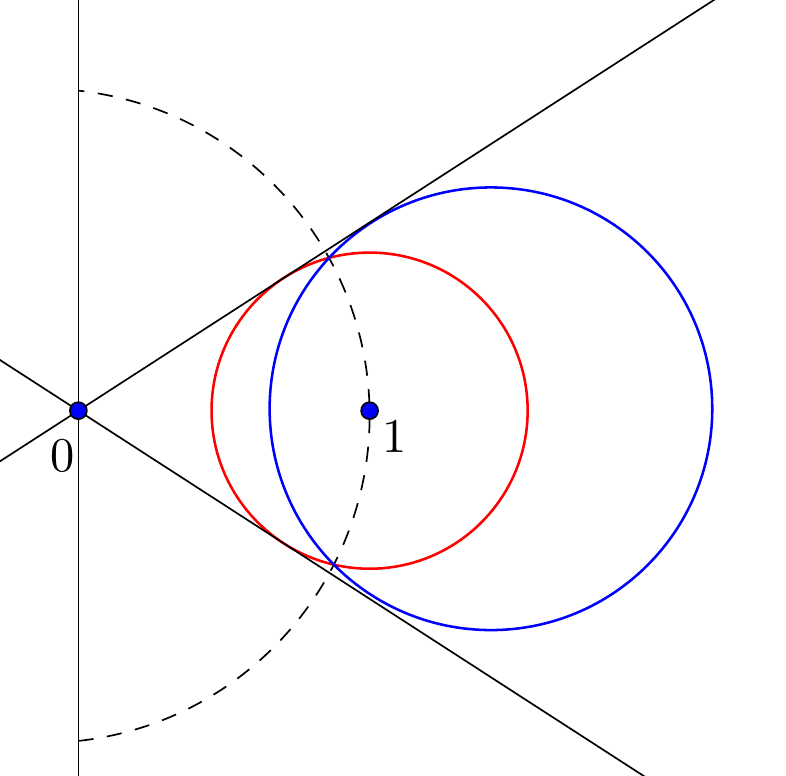}
\caption{$\mathcal{W}_\lambda$ is the union of the red and blue disks together with their convex hull.}
\label{fig:3point}
\end{figure}

In connection with the Nevanlinna-Pick interpolation, it is more customary to consider holomorphic maps into ${\DD}$. We may achieve this by considering $\Phi=\frac{\Psi-1}{\Psi+1}$ and the corresponding three-point problem. We shall consider the two problems in parallel as we will need to move back-and-forth between the two settings. 
As is well-known, in the one variable case, the positive semi-definiteness of Pick's matrix is equivalent to the existence of an interpolating map.
Agler \cite{agler} found an analogous necessary and sufficient condition for interpolation problems in the bidisk. The three-point case has been analysed in detail in \cite{agler-mccarthy}. In fact, Problem \ref{prob:threepoint} is very much related to Example 3.4 from \cite{agler-mccarthy}. The difference is that Example 3.4 assumes that $w>0$, while we need to consider complex values. Nevertheless, the inner functions found on \cite[p.~236]{agler-mccarthy} will be relevant.

\begin{proof}[Proof of Lemma \ref{lemma:NP}]
We go through the analysis based on \cite{agler-mccarthy}.
First, let us consider only the two-point problem given by the nodes $(\lambda_0,0)$ and $(0,\bar \lambda_0)$.
We must have $d_U(w,\bar w) \leqslant d_{\DD^2}((\lambda_0,0),(0,\bar \lambda_0))=d_{\DD}(0,k)$ in terms of the hyperbolic and Kobayashi distances. A short calculation reveals this is equivalent to the tangent cone displayed in Figure \ref{fig:3point}
\begin{equation}
\label{eq:cone}
 | \im w| \leqslant k |w|.
\end{equation}

Consider now the corresponding three-point problem in terms of $ \Phi=\frac{\Psi-1}{\Psi+1} \colon \DD^2 \to \DD$,
\[ (0,0) \mapsto 0,\quad (\lambda_0,0) \mapsto \zeta, \quad (0,\bar \lambda_0) \mapsto \bar \zeta, \qquad \zeta=\frac{w-1}{w+1}.
\]
We may assume that $\zeta \neq 0$ and by multiplying with a positive number that the problem is {\em extremal}. That is, $\|\Phi \|_\infty=1$ and there are no solutions with smaller norm. The problem is {\em non-degenerate} in the sense that there are no solutions that depend only on one of the coordinate functions (because $\zeta \neq 0$). It may happen that a sub two-point problem is already extremal, but assume for now, that this is not the case and the problem is a {\em genuine} three-point problem. Under these circumstances (extremal, non-degenerate, genuine) we have a unique solution which is a rational inner function of degree $2$ \cite[Theorem 0.1]{agler-mccarthy}. Moreover, the second order terms only involve the mixed product $\lambda \eta$ \cite[p. 234]{agler-mccarthy}. Since $\Phi(0,0)=0$, we have the following representation (see \cite{rudin})
\[ \Phi(\lambda,\eta)=\frac{\lambda \eta \, \overline{p(1/\bar \lambda,1/\bar \eta)}}{p(\lambda,\eta)},
\]
where $p$ is a linear polynomial which does not vanish on $\DD^2$,
\[ p(\lambda,\eta)= c_0+c_1 \lambda + c_2 \eta \neq 0.
\]
Since the function $\overline{\Phi(\bar \eta, \bar \lambda)}$ solves the same three-point problem and the solution is unique we have the symmetry
\[ \Phi(\lambda,\eta)=\overline{\Phi(\bar \eta, \bar \lambda)}.
\]
This leads to the following equations in terms of the coefficients
\[ c_0^2=\bar c_0^2, \quad c_0 c_1=\bar c_0 \bar c_2.
\]
There are two possibilities: either $c_0$ is real or purely imaginary. Observe, that we are free to multiply the coefficients by a real number without affecting the function $\Phi$. Thus the two cases can be normalized to
$c_0=2$ or $c_0=2i$. In the first case, we are led to
\[ p(\lambda,\eta)= 2+c\lambda+ \bar c \eta, \quad |c| \leqslant 1.
\]
The condition $|c| \leqslant 1$ is enforced because $p$ must not vanish on $\DD^2$.
Solving for $\Psi=\frac{1+\Phi}{1-\Phi}$ we arrive at
\begin{equation}
\label{eq:firstdisk} 
\Psi(\lambda,\eta)=\frac{1+\lambda \eta +c \lambda +\bar c \eta}{1-\lambda \eta}.
\end{equation}
In conclusion, $w=\Psi(\lambda_0,0)=1+c\lambda_0$ and so $|w-1| \leqslant k$. This is the first disk in \eqref{eq:two-disks}.

In the second case, when $c_0=2i$, we are led to  
\[ p(\lambda,\eta)=2i+c \lambda -\bar c \eta, \quad |c| \leqslant 1.
\]
This time, we have
\[ \Psi(\lambda,\eta)=\frac{1-\lambda \eta}{1-i c \lambda+i \bar c \eta +\lambda \eta}.
\]
In this case, $1/w=1/\Psi(\lambda_0,0)=1-ic\lambda_0$ and we now have $|1/w-1| \leqslant k$. This is the second disk in \eqref{eq:two-disks}. So far we covered the cases of genuine three-point (extremal) problems. Obviously, $\mathcal{W}_{\lambda_0}$ is a convex set, we will show that by taking the convex hull in \eqref{eq:two-disks} we also cover cases when the three-point problem degenerates to a two-point problem.
 
It may happen that a sub two-point problem is already extremal but only for the pair $(\lambda_0,0)$ and $(0,\bar \lambda_0)$ -- otherwise $\Phi$ would depend only on one variable.
From now on we assume that two-point problem corresponding to these nodes is extremal, that is we have equality in \eqref{eq:cone}. In such a situation $\Phi$ has to be a M\"obius transformation along an embedded analytic disk through the two nodes. By changing coordinates by an automorphism of $\DD^2$ we may arrange things so that the embedded disk is the diagonal $\{ (\lambda,\lambda) \colon \lambda \in \DD \}$ and the function is identity on the diagonal. This leads to the following situation in terms of the new holomorphic function $\tilde \Phi \colon \DD^2 \to \DD$
\[ \tilde \Phi(\lambda,\lambda)=\lambda, \quad \tilde \Phi(\bar \zeta, \zeta)=0 \quad \text{with } \zeta=\frac{w-1}{w+1}.
\]
We will show that necessarily $|\re \zeta| \leqslant |\zeta|^2$. Indeed, all such functions are of the following form (see \cite[(11.81)]{agler-mccarthy-book})
\[ \tilde \Phi(\lambda,\eta)=\frac{t \lambda+(1-t) \eta -\theta(\lambda,\eta) \lambda \eta}{1-\left( (1-t) \lambda +t \eta \right) \theta(\lambda,\eta)},
\]
where $t\in[0,1]$ and $\theta \in H^\infty_1(\DD^2)$. From $\tilde \Phi(\bar \zeta,\zeta)=0$ we find that
\[ t \bar \zeta+(1-t)  \zeta= \theta(\bar \zeta, \zeta) |\zeta|^2,
\]
and conclude with
\[|\re \zeta | \leqslant | t \bar \zeta+(1-t) \zeta | =|\theta(\bar \zeta, \zeta)| \cdot |\zeta|^2 \leqslant |\zeta|^2.
\]
In terms of $w$, we arrive at
\begin{equation}
\label{eq:2point}
 1-k^2 \leqslant \re w \leqslant 1,
\end{equation}
where we also used the fact that $|\im w|=k|w|$. 
Condition \eqref{eq:2point} corresponds exactly to the contribution of the convex hull in \eqref{eq:two-disks} as claimed. With all the cases analysed we concluded the proof of the lemma.

\end{proof}

\section{Holomorphic amplification}

Let $\,(\Omega,\sigma)\,$ be a measure space with its  $\,\mathscr L^p (\Omega,\sigma)\,$ spaces of complex-valued measurable functions. We will consider analytic families of measurable
functions in $\Omega$ parametrised by two variables. That is, jointly measurable functions $(x, (\lambda,\eta)) \mapsto \Phi_{(\lambda,\eta)} (x)$ defined on $\Omega \times \mathbb{D}^2$ outside of a measure zero set $E \subset \Omega$, for each fixed $x \in\Omega \setminus E$  the map $(\lambda,\eta) \mapsto \Phi_{(\lambda,\eta)}(x)$ is analytic in $\mathbb{D}^2$. The family is said to be \emph{non-vanishing} if $E$ can be chosen so that $\Phi_{(\lambda,\eta)}(x)\not=0$ 	for all $x \in\Omega\setminus E$ and for  all $(\lambda,\eta) \in \mathbb{D}^2$. 

We may use the holomorphic dependence to amplify a priori norm estimates in such families, see \cite[Interpolation Lemma]{AIPS}. With our application in mind, we will carry this out in the following setting.

\begin{proposition}
\label{prop:interpolation}
Suppose  $\;\{\Phi_{(\lambda,\eta) } \,;\, (\lambda,\eta) \in \DD^2\}$ is a holomorphic family of measurable functions, such that there is an exceptional set $E$ with $\sigma(E)=0$ and for every $(\lambda,\eta) \in \DD^2$, 
\begin{equation}
\label{eq:Phi-sym} 
  \Phi_{(\lambda,\eta)}(x)\not=0 \; \mbox{and } \;  \Phi_{(\lambda,\eta)}(x) = \overline{\Phi_{(\bar \eta, \bar \lambda)}(x)}, \quad \mbox{ for }  x \in \Omega \setminus E.
\end{equation}
Let $0 < p_0 < \infty$ and assume that   
\[ \Phi_{(0,0)} \equiv 1, \quad \mbox{and } \; \|\Phi_{(\lambda,\eta)} \|_{p_0} \leqslant 1.
\]
Then, for every $|\lambda|<1$ and for every complex exponent $t \in \C$ satisfying
\begin{equation} 
\label{eq:rangeoft}
|t|=\frac{p_0}{|\lambda|} \quad \mbox{and } \; \re t \geqslant p_0,
\end{equation}
we have 
\[ \int_\Omega \left| \Phi_{(\lambda,0)}^{\; t}  \right| \; \textrm{d} \sigma \leqslant 1.
\]
The choice of the continuous branch used here is determined by $\log \Phi_{(0,0)}\equiv 0$.
\end{proposition}

The proof goes along the lines of \cite[Lemma 4.1]{AIPS2}. 
The key difference is that we will make use of the three-point Nevanlinna-Pick problem analysed in Section \ref{se:NP} instead of the Schwarz lemma used in \cite{AIPS2}.

\begin{proof} 
By considering the analytic family $\Phi_\lambda^{\, p_0}$ we may restrict our attention to the $p_0=1$ case. Observe that our assumption $\|\Phi_{(0,0)}\|_1 \leqslant 1$ implies $\sigma(\Omega) \leqslant 1.$ Actually, by the maximum principle for analytic $\mathscr L^1$-valued functions, we may further assume the strict inequality
\begin{equation}\label{eq:size}
\sigma(\Omega) <1.
\end{equation}
Otherwise $\Phi_{(\lambda,\eta)}$ would be constant in the parameter $(\lambda,\eta)$, as we have
$\|\Phi_{(\lambda,\eta)}\|_1\leqslant 
1$ for all $(\lambda,\eta) \in \DD^2$.
We may also assume in the proof that $0<c\leqslant |\Phi_{(\lambda,\eta)}(x)|\leqslant C<\infty $
uniformly for all $(x,(\lambda,\eta))\in \Omega\times \mathbb{D}^2$,  as the reduction of the 
general situation to this is done similarly to \cite[Section 2]{AIPS}. We choose an arbitrary positive probability density $\wp$, uniformly bounded away from $0$ and $\infty$,
\[ \| \wp\,\|_1 = \int_\Omega \;\wp(x)\,\textrm{d}\sigma(x)\;= 1.
\]
By Jensen's inequality using the convexity of $x\mapsto x\log (x)$ and \refeq{eq:size} we have $I:=\int_\Omega \wp(x)\log \wp(x)\, dx>0.$
For a fixed $\,\wp\,$, we consider the holomorphic function $\Psi$ in the bidisk $\mathbb{D}^2$
\[
\Psi(\lambda,\eta)=\frac{1}{I} \int_\Omega \wp \log \Phi_{(\lambda,\eta)}  \,\textrm{d}\sigma\;.
\]
Again by Jensen's inequality  we have the bound
\begin{equation*}
\re \Psi(\lambda,\eta)-1 =\frac{1}{I} \int_\Omega \wp \log \frac{|\Phi_{(\lambda,\eta)}|}{\wp}\,\textrm{d}\sigma \leqslant \, \frac{1}{I} \log \left(\int_\Omega  |\Phi_{(\lambda,\eta)}| \, \textrm{d} \sigma\right)  \;\leqslant 0.
\end{equation*}
Thus $\Psi$ maps the bidisk into a half-plane $\re \Psi(\lambda,\eta) \leqslant 1$, while $\Psi(0,0)=0$ by our assumption $\Phi_{(0,0)} \equiv 1$. Furthermore, by our symmetry assumption in \eqref{eq:Phi-sym} $\Psi(\lambda,\eta)=\overline{\Psi(\bar \eta,\bar \lambda)}$.
At this stage we appeal to the three-point Nevanlinna-Pick problem of Problem \ref{prob:threepoint} and deduce that for any $|\lambda| <1$,
\[1-\Psi(\lambda,0) \in \mathcal{W}_{|\lambda|}.
\]
From now on we assume that the exponent $t \in \C$ satisfies \eqref{eq:rangeoft} (with $p_0=1$). That is,
\[ |t|=\frac{1}{|\lambda|} \quad \mbox{and } \; \re t \geqslant 1.
\]
For such $t$, we have $1-1/t \in \D \mathcal{W}_{|\lambda|}$. Moreover, the tangent line to the the convex set $t \cdot (1-\mathcal{W}_{|\lambda|})$ at the point $1$ is vertical. In other words, we have $\re(t w) \leqslant 1$ for any $w \in 1-\mathcal{W}_{|\lambda|}$. These statements follow from the description of $\mathcal{W}_{|\lambda|}$ in Lemma \ref{lemma:NP}.
In particular, 
\begin{equation}\label{eq:goal}
 \re \left(t \Psi(\lambda,0) \right) = \frac{1}{\int \wp \log \wp} \re \left(t \int \wp \log \Phi_{(\lambda,0)} \;\textrm{d}\sigma \right) \leqslant 1.
\end{equation}
Equivalently, we have 
$$\int_\Omega \wp \log \frac{\left| \Phi_{(\lambda,0)}^t \right|}{\wp}\,\textrm{d}\sigma \leqslant 0.
$$
By specialising the choice of $\wp$, that is,  choosing  $\wp(x):=\left| \Phi_{(\lambda,0)}(x)^t  \right|\left(\int_\Omega \big|\Phi_{(\lambda,0)}^t \big|\right)^{-1}$, 
we obtain our assertion in the form
\[
 \log \left( \int_\Omega \left| \Phi_{(\lambda,0)}^t  \right| \;\textrm{d}\sigma \right) \leqslant 0.
\]

\end{proof}

\section{Proof of the main result}

Towards the proof of Theorem \ref{thm:main}, we first prove a more technical version with various normalisations. In particular, we assume that the Beltrami coefficient $\mu$ is supported in the set $A_R=\{ z : |z|>R \}$ for some $R>1$.
 \begin{theorem} \label{thm:integrability}
Given $\delta>0$ and $0 \leqslant k < 1$, let $\mu$ be measurable,  $|\mu(z)| \leqslant (1-\delta) k \chi_{A_{R}}(z)$, with some $R >1$. Let $f  \in W^{1,2}_{loc}(\C)$ be the homeomorphic solution to  $\overline \partial f(z) =  \mu(z) \partial f(z)$ normalised by fixing $0,1,\infty$. We have 
 \begin{equation}
 \label{eq:intmeans11}
  \int_{|z| = 1}  \left| \left( z \frac{ \,  f'(z) \, }{f(z)} \right)^t \right| |dz|  \; \leqslant \frac{C(\delta)}{R-1},
 \end{equation}
 for any $t \in \mathbb{C}$ with $|t|=\frac{2}{k}$ and $\re t \geqslant 2$.
 Here $C(\delta) < \infty$ is a constant depending only on $\delta$. The complex powers in \ref{eq:intmeans11} are defined in terms of the holomorphic branch of $\log(z f'/f )$ which vanishes at $z=0$.
 \end{theorem}

\begin{proof} 
  We embed $f$ in a two-parameter holomorphic motion by setting for $(\lambda,\eta) \in \mathbb{D}^2$,
\begin{equation*}
\mu_{\lambda,\eta}(z)= \left\{ 
\begin{array}{r}
 \frac{\lambda}{k}\, \mu(z) \text{ for $|z|>1$},\\ 
\frac{\eta}{k}\, \overline{\mu (1/\bar z)} \text{ for $|z|<1$}.\\
\end{array} \right.\end{equation*}
Let $f_{\lambda,\eta}$ denote the unique solution to the Beltrami equation $f_{\overline z} = \mu_{\lambda,\eta} f_z$ normalised so that $f(0) = 0$, $f(1) = 1$ and $f(\infty) = \infty$. The uniqueness of the solution implies that $f_{k,0}=f$.  
 
We apply Proposition \ref{prop:interpolation} to the family
\begin{equation}
\label{perhe}
 \Phi_{(\lambda,\eta)}(z): = z \frac{ \, (f_{\lambda,\eta})'(z) \, }{f_{\lambda,\eta}(z)}, \qquad (\lambda,\eta) \in \DD^2, \, z \in \mathbb S^1.
\end{equation}
Since $f_{\lambda,\eta}$ is conformal on $\mathbb S^1$, this is a non-vanishing holomorphic family. 
We record the following symmetry of the construction,
\[ f_{\lambda,\eta}(z)=\frac{1}{\overline{f_{\bar \eta,\bar \lambda}(1/ \bar z)}}, \qquad z \in \mathbb{C}.
\]
This, in turn implies
\begin{equation}
\label{eq:symmetry}
\Phi_{(\lambda,\eta)}(z) = \overline{\Phi_{(\bar \eta, \bar \lambda)}(z)}, \qquad z \in \mathbb S^1.
\end{equation}

Standard quasiconformal estimates give the following global $L^2$-bounds, 
\begin{equation} 
(R-1) \,  \int_{|z|=1}  \left|  \frac{ \,  f_{\lambda,\eta}'(z) \, }{f_{\lambda,\eta}(z)} \right|^2 |dz| \; \leqslant C(\delta) < \infty.
\end{equation}
This is derived exactly as in \cite[(4.18)]{aipp}. The point is that $f_{\lambda,\eta}$ are $(1-\delta)$-quasiconformal for all $(\lambda,\eta) \in \DD^2$.

Observe that $\Phi_{(0,0)} \equiv 1$ since $f_{0,0}(z)=z$. 
We now use Proposition \ref{prop:interpolation} with $p_0=2$, $\Omega = \mathbb{S}^1$ and  $d\sigma(z) = c(\delta) (R-1) |dz|$ to deduce the conclusion of the theorem. A final remark is in order on the choices of complex logarithms. In Proposition \ref{prop:interpolation} we used the holomorphic flow to define complex logarithm. This is consistent with the statement of Theorem \ref{thm:integrability} as can be seen by considering the continuous function 
\[ (z,\lambda) \mapsto \log  \left( z \frac{ \, (f_{\lambda,0})'(z) \, }{f_{\lambda,0}(z)} \right)
\]
on the simply connected region $\{z \colon |z|<R \} \times \{ \lambda \colon |\lambda| <1\}$.
\end{proof}

In order to compare the growth of complex powers of $f'$ with that of $zf'/f$ we need to control the growth of $\log (f(z)/z)$. The growth is quite slow, for a bounded function $f \in \mathcal S$, we have \cite[Section 4]{baranov-hedenmalm}
\[\left|  \log \frac{f(z)}{z} \right|= O \left( \sqrt{\log \frac{1}{1-|z|^2}} \right), \quad |z| \to 1.
\] 
This would suffice for our purposes but it turns out that in our case, when $f$ has a quasiconformal extension, $\log (f(z)/z)$ is even bounded.
  
\begin{lemma}
\label{lemma:foverz}
Let $f \colon \DD \to \Omega$ be a conformal map in the class $\mathcal S_k$. Then $\log(\frac{f(z)}{z})$ is a bounded function. The bound only depends on the constant $k$.
\end{lemma}

\begin{proof}
Embed $f$ in the holomorphic motion by solving the Beltrami equation for $\lambda \in \DD$
\[ \bar \D f_\lambda = \frac{\lambda}{k} \frac{\bar \D f}{\D f} \, \D f_\lambda, \quad f_\lambda(0)=0,\  f_\lambda'(0)=1. 
\]
For a fixed $z \in \DD$, consider the holomorphic function $g(\lambda)= \log \frac{f_\lambda(z)}{z}$.
From distortion properties of quasiconformal maps the real part is bounded from above and below (independently of $z$),
\[ -C(|\lambda|) \leqslant \re g(\lambda)=\log \left| \frac{f_\lambda(z)}{z} \right| \leqslant C(|\lambda|).
\]
It follows automatically that the imaginary part also satisfies a similar bound. Indeed, let $\rho=\sqrt{k}$ and $\DD_\rho=\{z \colon |z| <\rho\}$. The function $g$ maps $\DD_\rho$ into a strip $S_\rho=\{z \colon -C(\rho) <\re z < C(\rho) \}$ with $g(0)=0$. In view of the Schwarz-Pick lemma, $g$ contracts in the hyperbolic metric $d_{S_\rho}(0,g(k)) \leqslant d_{\DD_\rho}(0,k)=d_\DD(0,\sqrt{k})$. We deduce that
\[ \left| \log \frac{f(z)}{z} \right| =|g(k)| \leqslant C(k),
\]
with a different constant.

\end{proof}

We are ready to deduce Theorem \ref{thm:main}.
\begin{proof}[Proof of Theorem \ref{thm:main}]
Consider a conformal map $f \colon \DD \to \Omega$ with a $k$-quasi\-conformal extension. By composing with a similarity transformation, we may assume that $f \in \mathcal{S}_k$. Let $t \in \C$ with $|t|=\frac{2}{k}$ and $\re t \geqslant 2$. By Lemma \ref{lemma:foverz}
\begin{align*} 
\int_{|z|=r} |f'(z)^t| |dz| & \leqslant C(k,t) \int_{|z|=r}  \left| \left( z \frac{ \,  f'(z) \, }{f(z)} \right)^t \right| |dz|
\end{align*}
We continue the inequality by applying Theorem \ref{thm:integrability} for the function $f(rz)/f(r)$. For a small $\delta>0$ we have
\begin{align*}
& \leqslant C(k,t) (1-r)^{-\delta |t|} \int_{|z|=r} \left| \left( z \frac{ \,  f'(z) \, }{f(z)} \right)^{(1-\delta)t}  \right| |dz|  \\
& \leqslant C(k,t,\delta) (1-r)^{-1-\delta |t|}.
\end{align*}
This shows that $\beta_f(t) \leqslant 1+\delta |t|$, and letting $\delta \to 0$ proves the Theorem.
\end{proof}

Integration in polar coordinates (and the fact that $B_k(t)$ is convex in every direction) gives the following corollary.

\begin{corollary}
\label{cor:integrability}
Let $f \colon \DD \to \C$ be a conformal map with $k$-quasiconformal extension. Then we have
\[ \int_{\mathbb D} \left| (f\, ')^t \right| < \infty \qquad \mbox{for} \quad |t| <\frac{2}{k},\ \re t \geqslant k |t|. 
\]
\end{corollary}

\begin{remark}
For any $t \in \C$ with $|t|=\frac{2}{k}$ there exist a conformal map $f$ with $k$-quasiconformal extension such that 
\[ \int_{\mathbb D} \left| (f\, ')^t \right| = \infty.
\]
Namely, it is sufficient to consider Example \ref{ex:basicexamples} with $\sigma=1-\frac{2}{t}$.
\end{remark}

\section{Example: welding of radial stretchings}
\label{se:example}

It will be instructive to analyse Examples \ref{ex:basicexamples} from the point of  view of holomorphic motions. We will embed them in a motion parametrised by the bidisk as in the proof of Theorem \ref{thm:integrability}. Again, it will be more convenient to work with upper and lower half planes $\C_+$ and $\C_-$ instead of the unit disk and its exterior. Set the Beltrami coefficient $\mu(z)=z/\bar z$, $z \in \C_+$.
We now describe the holomorphic motion $f_{\lambda,\eta}$ solving the normalised Beltrami equation with coefficient 

\begin{equation*}
\mu_{\lambda,\eta}(z)= \left\{ 
\begin{array}{r}
 \lambda \mu(z) \text{ for $z \in \mathbb{C_+}$},\\
\eta \overline{\mu (\bar z)} \text{ for $z \in \mathbb{C_-}$},\\
\end{array} \right.\end{equation*}
for $(\lambda,\eta) \in \DD^2$.
Along the diagonal $(\lambda,\lambda) \in \DD^2$ the solution is the radial stretching/twisting map
\[ f_{\lambda,\lambda}(z)=\frac{z}{|z|} \, |z|^{\frac{1+\lambda}{1-\lambda}}.
\]
In general, the solution $f_{\lambda,\eta}$ will be given in $\C_+$ by $f_{\lambda,\lambda}$ followed by a (conformal) complex power map and in $\C_-$ by $f_{\eta,\eta}$ followed by another complex power map.
We have to choose the complex powers in such a way that the map in $\C_+$ and in $\C_-$ matches on the real line. Explicitly, with the notation $\sigma_+=\frac{1+\lambda}{1-\lambda}$, $\sigma_-=\frac{1+\eta}{1-\eta}$ and $1/\sigma=\frac{1}{2}(1/\sigma_+ + 1/\sigma_-)$,
\[ f_{\lambda,\eta}(z)= \left( \frac{z}{|z|} |z|^{\sigma_\pm} \right)^{\sigma/\sigma_{\pm}} \text{ for $z \in \mathbb{C_\pm}$}.
\]
The stretching/twisting behaviour of the motion $f_{\lambda,\eta}$ (at the origin) is described by the complex exponent $\sigma(\lambda,\eta)$ where
\[ \frac{1}{\sigma(\lambda,\eta)}= \frac{1}{2} \left( \frac{1-\lambda}{1+\lambda} + \frac{1-\eta}{1+\eta} \right).
\]

\begin{remark}
Observe that $\sigma \colon \DD^2 \to U$ satisfies the conditions of Problem \ref{prob:threepoint} and $\sigma$ is equal to $\Psi$ from \eqref{eq:firstdisk} with the choice of $c=1$ (general values of $c$ may be covered by reparametrising the motion).
That is, we see that the first disk in Lemma \ref{lemma:NP} is a `physical example', in the sense that
it appears from holomorphic functions built from holomorphic motion of conformal maps (through the constructions in the proofs of Theorem \ref{thm:integrability} and Proposition \ref{prop:interpolation}).
Conjecture \ref{conj:circular-brennan} amounts to saying that this first disk is the only allowable region for such physical examples.
Our method does not capture subtler properties of conformal maps and, in particular, is unable to rule out the `non-physical example' of the second disk in Lemma \ref{lemma:NP}. 
\end{remark}


\section{Twisting estimate for quasidisks}

As an application of our main result, we give a geometric multifractal dimension estimate on the size of twisting for a quasidisk. Let $\Omega \subset \C$ be a bounded $L$-quasidisk and consider points $x \in \D \Omega$ which twist at a prescribed rate $\gamma \in \R$.
The fact that $\Omega$ is the image of the unit disk under a global $L$-quasiconformal map is equivalent to the fact that (any) conformal map $f \colon \DD \to \Omega$ has a $k$-quasiconformal extension. The precise relation of the constants involved are \cite{kuhnau,smirnov}
\[ k=\frac{L^2-1}{L^2+1}.
\]
We now recall the notion of rate of twisting  from Section 2 in terms of the conformal map $f \colon \DD \to \Omega$. Let $x=f(\zeta)$, $\zeta \in \D \DD$ and consider the limit
\begin{equation}
\label{eq:gammaspiralling}
\lim_{\tau \to 1} \frac{\arg(f(\tau \zeta)-f(\zeta))}{\log |f(\tau \zeta)-f(\zeta)|}.
\end{equation}
If this limit exists and is equal to $\gamma \in \R$, then we say that $x \in \D \Omega$ is $\gamma$-spiralling (or twisting). 

\begin{theorem} 
\label{thm:twisting}
Let $f \colon \DD \to \Omega$ be a conformal map (onto) with $k$-quasiconformal extension. The set $F \subset \D \Omega$ of $\gamma$-spiralling points has Hausdorff dimension at most
\begin{equation}
\Hdim F \leqslant 2-\frac{2\sqrt{1-k^2}}{k} |\gamma|.
\end{equation}
\end{theorem}

The reader may contrast this with \cite[Corollary 5.4]{AIPS2} which gives a similar multifractal estimate for twisting points. That bound applies to arbitrary quasiconformal maps, and therefore gives a weaker result when applied to the map $f$. By Proposition \ref{prop:propk}, the maximal pointwise twisting for a quasdisk is $|\gamma| = k/\sqrt{1-k^2}$. Theorem \ref{thm:twisting} is effective in the sense that it shows this extremal twisting can only happen on a set of dimension zero.

\begin{proof}[Proof of Theorem \ref{thm:twisting}]
Let $E=f^{-1}(F)$ be the pre-image of $\gamma$-spiralling points and fix an $\varepsilon >0$. By definition of \eqref{eq:gammaspiralling} we may select for any $\zeta \in F$ a radius $r_\zeta \in (0,\varepsilon)$ such that 
\begin{equation}
\label{eq:gammaspiralling}
 \frac{\arg(f(\tau \zeta)-f(\zeta))}{\log |f(\tau \zeta)-f(\zeta)|} \in (\gamma-\varepsilon,\gamma+\varepsilon),
\end{equation}
for $\tau \geqslant 1-r_\zeta$.
Let us denote the arcs of $\D \DD$ with center $\zeta$ and length $r_\zeta$ by $I_\zeta$.
Vitali's covering lemma allows us to select a countable disjoint subcollection $\{I_{\zeta_i} \}$ such that 
$E \subset \cup \, 5 I_{\zeta_i}$.
For each $\zeta_i$, we associate a ``top-half'' of  a Carleson box $Q_i$ as follows.
Let $Q_i=\{z : z/|z| \in I_i \mbox{ and } 1-r_i \leqslant |z| \leqslant 1-r_i/2\}$, where $r_i=r_{\zeta_i}$ and $I_i=I_{\zeta_i}$ for short.
Now we show that $\arg f'$ has to be big (in absolute value) on each $Q_i$. 
From Koebe distortion and quasisymmetric properties of $f$ we have that, for $z \in Q_i$
\[ \left| \re \left( \log f'(z) -  \log \frac{f((1-r_i) \zeta_i)-f(\zeta_i)}{r_i} \right) \right| \leqslant C(k).
\] 
From this, by an argument exactly as in Lemma \ref{lemma:foverz}, we deduce that the imaginary part
also satisfies a similar bound
\begin{equation*}
\label{eq:argfprime}
 \left| \arg f'(z) - \arg (f((1-r_i) \zeta_i)-f(\zeta_i)) \right| \leqslant C(k), \quad z \in Q_i.
\end{equation*}
Combined with \eqref{eq:gammaspiralling} we conclude that 
\[ \arg f'(z) \leqslant (\gamma-\varepsilon) \log \tilde r_i+C(k), \quad z \in Q_i,
\]
with the notation $\tilde r_i=|f((1-r_i)\zeta_i)-f(\zeta_i)|$. We also assume here $\gamma>0$ for simplicity (the case $\gamma<0$ is similar, we just have to consider $s<0$ below).
For $s>0$,
\[ \int_{Q_i} |(f')^{2+is}| = \int_{Q_i} |f'|^2 \exp(-s \arg f') \geqslant C(k) \tilde r_i^{2-(\gamma-\varepsilon)s}.
\]
Summing it up over $Q_i$, we obtain
\[ \sum_i (\tilde r_i)^{2-(\gamma-\varepsilon) s} \leqslant C(k)  \int_{\cup Q_i} |(f')^{2+is}| \leqslant C(k) \int_\DD |(f')^{2+is}| <\infty,
\]
as long as $|s|<2\frac{\sqrt{1-k^2}}{k}$, see Corollary \ref{cor:integrability}.
The sets $\{ f(5I_i) \}$ provide a cover of $F$, each with diameter comparable to $\tilde r_i$. This shows that $\Hdim F \leqslant 2-(\gamma-\varepsilon)s$. Finally, taking the limits $s \to 2\frac{\sqrt{1-k^2}}{k}$ and $\varepsilon \to 0$ proves the theorem.

\end{proof}

\bibliographystyle{amsplain}

\end{document}